\documentclass[a4paper,12pt,reqno]{amsart}
\usepackage{a4wide}
\usepackage{amsmath}
\usepackage{amssymb}
\usepackage{amsthm}
\usepackage{latexsym}
\usepackage{graphicx}
\usepackage{tikz}
\usepackage[english]{babel}
\usepackage{calc}
\usepackage{accents}

\newtheorem{thm}[subsection]{Theorem}
\newtheorem{lem}[subsection]{Lemma}
\newtheorem{cor} [subsection]{Corollary}
\theoremstyle{definition}
\newtheorem{dfn} [subsection]{Definition}
\theoremstyle{remark}
\newtheorem{obs} [subsection]{Remark}
\newtheorem{exm} [subsection]{Example}

\def\supp{\operatorname{supp}}
\def\deg{\operatorname{deg}}
\def\lcm{\operatorname{lcm}}

\def\reg{\operatorname{reg}}

\newcommand{\dbtilde}[1]{\accentset{\approx}{#1}}

\def\k {\mathrm{k}}

\def\mm {\mathfrak{m}}

\DeclareMathOperator{\NN}{\mathbb {N}}

\DeclareMathOperator{\QQ}{\mathbb {Q}}

\DeclareMathOperator{\pd}{pd}
\DeclareMathOperator{\Tor}{Tor}

\numberwithin{equation}{section}

\begin{document}

\title[Betti numbers of powers of path ideals of cycles]{Betti numbers of powers of path ideals of cycles}
\author[Silviu B\u al\u anescu, Mircea Cimpoea\c s, and Thanh Vu]{Silviu B\u al\u anescu$^1$, Mircea Cimpoea\c s$^2$ and Thanh Vu$^3$}  
\date{}

\keywords{Betti numbers, Monomial ideal, Path ideal}

\subjclass[2020]{13C15, 13P10, 13F20}

\footnotetext[1]{ \emph{Silviu B\u al\u anescu}, National University of Science and Technology Politehnica Bucharest, Faculty of
Applied Sciences, 
Bucharest, 060042, E-mail: silviu.balanescu@stud.fsa.upb.ro}
\footnotetext[2]{ \emph{Mircea Cimpoea\c s}, National University of Science and Technology Politehnica Bucharest, Faculty of
Applied Sciences, 
Bucharest, 060042, Romania and Simion Stoilow Institute of Mathematics, Research unit 5, P.O.Box 1-764,
Bucharest 014700, Romania, E-mail: mircea.cimpoeas@upb.ro,\;mircea.cimpoeas@imar.ro}
\footnotetext[3]{ \emph{Thanh Vu}, Institute of Mathematics, VAST, 18 Hoang Quoc Viet, Hanoi, Vietnam 122300, E-mail: vuqthanh@gmail.com}

\begin{abstract}  
Let $J_{n,m} = (x_1\cdots x_{m},x_2 \cdots x_{m+1},\ldots,x_{n}x_1\cdots x_{m-1})$ be the $m$-path ideal of a cycle of length $n \ge 5$ over a polynomial ring $S = \k[x_1,\ldots,x_n]$. Let $t\geq 1$ be an integer. We show that $J_{n,m}^t$ has a linear free resolution and give a precise formula for all of its Betti numbers when $m = n-1, n-2$.
\end{abstract}

\maketitle

\section{Introduction}

Let $S = \k[x_1,\ldots,x_n]$ be a standard graded polynomial ring over a field $\k$. For a finitely generated graded $S$-module $M$, the $(i,j)$-graded Betti number of $M$, denoted by $\beta_{i,j}(M)$ is defined by 
$$\beta_{i,j}(M) = \dim_\k \Tor_i^S(\k,M)_j.$$

The Betti numbers of a homogeneous ideal $I$, among the most important invariants of $I$, capture many geometric properties of the projective variety defined by $I$ (see \cite{E} for more information). Given $I$, computing all of its Betti numbers is always a challenging but interesting problem. We have very few classes of ideals for which we know all of their Betti numbers. When considering the powers of $I$, much less is known. A celebrated result of Akin, Buchsbaum, and Weyman \cite{ABW} describes the minimal free resolutions of powers of the maximal minors of generic matrices. Even in the case of edge ideals of graphs, except complete graphs or complete bipartite graphs, the Betti numbers of their powers are unknown. The coarser invariants, Castelnuovo-Mumford regularity, of powers of edge ideals of graphs are better known and still an active area of research \cite{HHZ, B, BHT, MV2, BN} (see \cite{MV1} for a recent survey on the topic). On the other hand, the projective dimension of powers of edge ideals of graphs is much more difficult to compute. The projective dimensions of powers of edge ideals of paths and cycles are only given very recently in \cite{BC2, MTV}, and some classes of trees in \cite{MTV, HHV}.

Constructing minimal free resolutions of monomial ideals and their powers have regained interest recently \cite{EFSS, CEFMMSS1, CEFMMSS2}, where the authors look at the free resolutions of monomial ideals from the perspective of cellular resolutions. In this work, we not only give formulae for the Betti numbers of powers of some path ideals of cycles but also give a construction of their minimal free resolutions via the mapping cone constructions.

Let us now recall the main object of study in this work. The \emph{cycle} of length $n$, where $n\ge 3$, is the (simple) graph $C_n$ on the vertex set 
$V(C_n)=\{1,2,\ldots,n\}$ and the edge set 
$$E(C_n)=\{ \{1,2\},\ldots, \{n-1,n\}, \{n,1\}\}.$$ 
For an integer $m$ with $2 \le m \le n$, the $m$-path ideal of $C_n$ is
$$J_{n,m}:=(x_1\cdots x_m,\ldots,x_{n-m+2}\cdots x_nx_1,\ldots,x_nx_1\cdots x_{m-1})\subset S.$$
In \cite{AF}, Alilooee and Faridi computed all the graded Betti numbers of path ideals of cycles and lines. The problem of calculating all the Betti numbers of the powers of general path ideals of cycles is very complicated, which can be seen partially in the work of \cite{BC2, MTV}. In this work, we study the first two non-trivial instances of the problem, namely the $n-1$-path ideals of $n$-cycles and $n-2$-path ideals of $n$-cycles. Our main results are as follows.

\begin{thm}\label{teo1}
For all $n\geq 2$ and $t\geq 1$, the ideal $J_{n,n-1}^t$ has a linear free resolution and 
$$ \beta_i (J_{n,n-1}^t ) = \binom{n-1}{i}\binom{n+t-i-1}{t-i},\text{ for all }i\geq 0.$$
\end{thm}

For the case of $n-2$-path ideals of $n$-cycles, we give recursive formulae for computing the Betti numbers of their powers (see 
Lemma \ref{lem_b_c} and Lemma \ref{lem_c_b}) and deduce the following:

\begin{thm}\label{teo2}
    For all $n \ge 3$ and all $t \ge 1$, the ideal $J_{n,n-2}^t$ has a linear free resolution and 
    $$\pd (J_{n,n-2}^t) = \begin{cases}
        \min \{ n-1, 2t\} & \text{ if } n \text { is odd}, \\
        \min \{ n-2, 2t\} & \text{ if }n  \text{ is even}.
    \end{cases}$$
\end{thm}

\begin{thm}\label{teo3}
Let $k\geq 1$ be an integer. 
\begin{enumerate}
\item[(1)] For all $i\geq 0$, $t\ge 0$ and $n = 2k+1$ we have:
$$ \beta_i(J_{n,n-2}^t) = \sum_{j=0}^{\left\lfloor \frac{i}{2} \right\rfloor} \binom{n}{i-2j}\binom{n+t-1-i+j}{n-1} 
- \sum_{j=0}^{\left\lfloor \frac{i-1}{2} \right\rfloor} \binom{n}{i-1-2j}\binom{t+k-1-j}{n-1}. $$
\item[(2)] For all $i\geq 0$, $t\ge 0$ and $n = 2k$ we have:
$$ \beta_i(J_{n,n-2}^t) = \sum_{j=0}^{\left\lfloor \frac{i}{2} \right\rfloor} \binom{n}{i-2j}\binom{n+t-1-i+j}{n-1} 
-  \sum_{j=0}^{\left\lfloor \frac{i}{2} \right\rfloor} \binom{n}{i-2j}\binom{t+k-1-j}{n-1}.$$
\end{enumerate}
\end{thm}

To prove the main results, we establish various Betti splittings involving powers of $n-1$ and $n-2$-path ideals of $n$-cycles. These Betti splittings give rise to recursive equations for the Betti numbers. In the case of the $n-1$-path ideal, we deduce a formula for the generating function of the Betti numbers. In the case of $n-2$-path ideal, to prove Theorem \ref{teo3}, we show that both sides satisfy the same set of recurrent relations and agree at the boundary terms. See Section \ref{sec_experiment} for more detail. To illustrate the effectiveness of the results, we give an example.

\begin{exm} By Theorem \ref{teo2} and Theorem \ref{teo3}, the Betti table of $J_{27,25}^4$ is as follows.

\begin{table}[h!]

    \[\begin{array}{l|c c c c  c c c c c}
            &   0   & 1     & 2     & 3 & 4 & 5 & 6 & 7 & 8\\ \hline
        -   &   -   &   -   & -     & - & - & -  & - & - & -\\
      100   &27405    &  98658 & 136332 & 89181& 27405   & 3654 & 378 & 27 & 1
    \end{array}
\]
\label{tab:table1}
\end{table}
\end{exm}
We now outline the organization of the paper. In Section \ref{sec_pre}, we recall the notion of Betti splittings and prove a key result to compute the intersection of monomial ideals. In Section \ref{sec_n_1}, we prove Theorem \ref{teo1} which computes all the Betti numbers of powers of $J_{n,n-1}$. In Section \ref{sec_n_2}, we derive recurrences for the Betti numbers of powers of $J_{n,n-2}$ and establish Theorem \ref{teo2}. In Section \ref{sec_experiment}, we prove Theorem \ref{teo3}.

\section{Preliminaries} \label{sec_pre}

Throughout this section, we let $S = \k[x_1,\ldots, x_n]$ be the polynomial ring over an arbitrary field $\k$, with the standard grading.

\subsection*{Projective dimension and regularity} 

Let $M$ be a finitely generated graded $S$-module and $i,j$ two integers with $i \ge 0$.
The $(i,j)$-graded Betti number of $M$ is defined by 
$$\beta_{i,j}(M) = \dim_\k \Tor_i^S (\k,M)_j.$$
The $i$-th Betti number of $M$ is 
$$\beta_i(M) = \dim_\k \Tor_i^S(\k,M) = \sum_j \beta_{i,j}(M).$$ 
The projective dimension of $M$, denoted by $\pd_S(M)$, and the Castfelnuovo-Mumford regularity of $M$, denoted by $\reg_S(M)$, are defined as follows:
\begin{align*}
    \pd_S (M) &= \sup \{i :\; \beta_i(M) \neq 0\},\\
    \reg_S(M) &= \sup \{ j -i :\; \beta_{i,j} (M) \neq 0\}.
\end{align*}

We have the following elementary facts:

\begin{lem}\label{lem_mul_x} 
Let $x_j$ be a variable and $I$ a nonzero homogeneous ideal of $S$. Then 
\begin{enumerate}
    \item $\beta_i(x_jI)=\beta_i(I)$, for all $i\ge 0$.
    \item $\pd_S(x_jI) = \pd_S(I)$.
\end{enumerate}    
\end{lem}

\subsection*{Path ideals of graphs}
Conca and De Negri \cite{CD} introduced the notion of $t$-path ideals of graphs as a generalization of the notion of edge ideals of graphs. Let us now recall this definition. Let $G$ denote a finite simple graph over the vertex set $V(G) = \{1,\ldots,n\}$ and the edge set $E(G)$. Let $t \ge 2$ be a natural number. A $t$-path of $G$ is a sequence of distinct vertices $i_1, \ldots, i_t$ of $G$ such that $\{i_1,i_2\}$,$\ldots, \{i_{t-1},i_t\}$ are edges of $G$. The $t$-path ideal of $G$ is defined to be
$$I_t(G)= \left ( x_{i_1} \cdots x_{i_t} \mid i_1, \ldots,i_t \text{ is a } t\text{-path of } G \right ) \subseteq S.$$

In contrast to the edge ideals of graphs, not much is known about the homological invariants of powers of $t$-path ideals of graphs when $t > 2$. In \cite{BC1, SL, SWL}, the authors gave formulae for the depth and Stanley depth, regularity, and multiplicity of powers of path ideals of path graphs, respectively. In \cite{BC2}, the authors obtained partial results for the depth and Stanley depth of powers of edge ideals of cycles.

\subsection*{Betti splittings} 

Betti splittings of monomial ideals were first introduced by Francisco, Ha, and Van Tuyl in \cite{FHV}, motivated by the work of Eliahou and Kervaire \cite{EK}. This notion has regained interest recently in several works \cite{CF, HV}. We recall the definition and the following results about Betti splittings, following \cite{NV2}.

\begin{dfn}
Let $P,I,J$ be proper nonzero homogeneous ideals of $S$ with $P = I + J$. 
The decomposition $P = I +J$ is called a \emph{Betti splitting} if for all $i \ge 0$ we have $$\beta_i(P) = \beta_i(I) + \beta_i(J) + \beta_{i-1}(I \cap J).$$
\end{dfn}

\begin{lem}\label{lem_pd_reg_split} Assume that $P = I + J$ is a Betti splitting of ideals of $S$. Then 
\begin{align*}
    \pd_S(P) &= \max \{ \pd_S(I), \pd_S(J), \pd_S(I \cap J) + 1\},\\
    \reg_S(P) &= \max\{ \reg_S(I), \reg_S(J), \reg_S(I \cap J) - 1\}.    
\end{align*}
\end{lem}
\begin{proof}
    See \cite[Corollary 2.2]{FHV} or \cite[Lemma 3.7]{NV2}.
\end{proof}

\begin{dfn} Let $\varphi: M \to N$ be a morphism of finitely generated graded $S$-modules. We say that $\varphi$ is $\Tor$-vanishing if for all $i \ge 0$, we have $\Tor_{i}^S(\k,\varphi) = 0$.    
\end{dfn}

We have the following criterion for Betti splitting by Nguyen and Vu \cite[Lemma 3.5]{NV2}.

\begin{lem}\label{lem_splitting_criterion_1} Let $I,J$ be nonzero homogeneous ideals of $S$ and $P = I+J$. The decomposition $P = I +J$ is a Betti splitting if and only if the inclusion maps $I \cap J \to I$ and $I\cap J \to J$ are $\Tor$-vanishing.    
\end{lem}

In particular, we have

\begin{lem}\label{lem_splitting_criterion_2} Let $I, J$ be homogeneous ideals of $S$ and $P = I + J$. Assume that $I$ and $J$ have a linear free resolution and $I \cap J \subseteq \mm I$ and $I \cap J \subseteq \mm J$, where $\mm$ is the maximal homogeneous ideal of $S$. Then, the decomposition $P = I + J$ is a Betti splitting.    
\end{lem}
\begin{proof}
Let $d$ be the degree of minimal generators of $I$. By the assumption, we have $\Tor_i^S(k,I \cap J)_{i+d} = 0$ for all $i$. Since $I$ has a linear free resolution, the inclusion map $I\cap J \to I$ is $\Tor$-vanishing. Similarly, the inclusion map $I \cap J \to J$ is $\Tor$-vanishing. The conclusion follows from Lemma \ref{lem_splitting_criterion_1}.
\end{proof}

\begin{obs}
    By \cite[Proposition 2.1]{FHV} and \cite[Lemma 4.4]{NV1}, once we have $P = I + J$ is a Betti splitting, then the mapping cone construction for the map $I \cap J \to I \oplus J$ yields a minimal free resolution of $P$.
\end{obs}

\subsection*{An intersection of monomial ideals} In this subsection, we give a simple but useful lemma for computing the intersection of monomial ideals of certain forms. First, we define the support of a monomial ideal.

For a monomial $u\in S$, the \emph{support} of $u$, denoted by $\supp(u)$ is the set of variable $x_i$ such that $x_i|u$. Let $J\subset S$ be a monomial ideal of $S$ with the minimal monomial generating set $G(J)=\{u_1,\ldots,u_m\}$. The \emph{support} of $J$ is defined by 
$$\supp(J)=\bigcup_{i=1}^m \supp(u_i).$$

We have
\begin{lem}\label{lem_intersection} 
Assume that $J \subseteq K$ are monomial ideals of $S$ such that $x_n \notin \supp(J)\cap \supp(K)$ and let $I = J + x_nK$. Then for any 
$s\ge 0$ and $t\ge 1$, we have  
$$(x_nK)^s J^t \cap (x_nK)^{s+1} I^{t-1} = x_n (x_nK)^s J^t,\text{ where }I^0=J^0=S.$$
\end{lem}

\begin{proof} 
Let $f$ be a minimal monomial generator of $(x_nK)^s J^t \cap (x_nK)^{s+1} I^{t-1}$. Then there exist $g\in G((x_nK)^s J^t)$ and $h\in G((x_nK)^{s+1} I^{t-1})$ 
such that $f = \lcm(g,h)$. Since $x_n \notin \supp(J)\cap \supp(K)$, we deduce that $\deg_{x_n}(g) = s$ and $\deg_{x_n}(h) \geq s+1$. Therefore $x_ng\mid f$, that is $f \in x_n(x_nK)^s J^t$. 

Conversely, we have
$$x_n (x_nK)^s J^t = (x_nJ) (x_nK)^s J^{t-1} \subseteq (x_nK)^{s+1} J^{t-1} \subseteq (x_nK)^{s+1} I^{t-1}.$$
The conclusion follows.
\end{proof}

\section{Betti numbers of powers of $(n-1)$-path ideals of $n$-cycles}\label{sec_n_1}

In this section, we compute all the Betti numbers of powers of $(n-1)$-path ideals of $n$-cycles. 
We denote $f_1 = x_1 \cdots x_{n-1}, \ldots, f_n = x_nx_1 \cdots x_{n-2}$. If $n\geq 3$, the ideal 
$$I_n =J_{n,n-1}=(f_1,\ldots,f_n)$$ is
the $(n-1)$-path ideal of a cycle of length $n$. Also, we let $I_1=K[x_1]$ and $I_2=(x_1,x_2)$. 

Assume that $n\ge 3$. Note that 
$$I_n = f_1 + x_n I_{n-1}\text{ and }f_1 \in I_{n-1}.$$
By applying Lemma \ref{lem_intersection} to $(f_1)\subset I_{n-1}$, we get the following result:

\begin{lem}\label{lemi}
For all $s\ge 0$ and $t\ge 1$ we have
$$ I_{n-1}^s(f_1^t) \cap x_nI_{n-1}^{s+1} I_n^{t-1} = x_n I_{n-1}^s(f_1^t).$$
\end{lem}

Now, we prove a key lemma for computing all the Betti numbers of powers of $I_n$.

\begin{lem}\label{lem_first_case} Assume that $n \ge 3$. For all $s \ge 0$ and $t\ge 0$, the ideal $J = I_{n-1}^s I_n^t$ has a linear free resolution and, if $t\ge 1$, the decomposition 
$$I_{n-1}^s I_n^t = I_{n-1}^s (f_1^t) + x_n I_{n-1}^{s+1} I_n^{t-1}$$ is a Betti splitting.
\end{lem}
\begin{proof} 
We prove by induction on $n\geq 3$ and $t\geq 0$. Assume $n=3$. We have that $I_{2}=(x_1,x_2)$, $f_1=x_1x_2$ and 
$I_3=(f_1)+x_3I_2=(x_1x_2,x_2x_3,x_3x_1)$. Therefore 
$$J=I_2^s I_3^t=(x_1,x_2)^s (x_1x_2,x_2x_3,x_3x_1)^t.$$
If $t=0$ then $J=(x_1,x_2)^s$ has linear free resolution. If $t\ge 1$, we consider the decomposition
$$ J=I_2^s I_3^t = I_2^s (x_1x_2)^t +  x_3 I_2^{s+1} I_3^{t-1}.$$
By induction hypothesis on $t$, the ideal $I_2^{s+1} I_3^{t-1}$ has a linear free resolution and thus $x_3 I_2^{s+1} I_3^{t-1}$ has 
a linear free resolution also. Similarly, $I_2^s (x_1x_2)^t$ has a linear free resolution. By Lemma \ref{lemi} we have
$$ I_{2}^s(x_1x_2)^t \cap x_3I_{2}^{s+1} I_3^{t-1} = x_3 I_{2}^s(x_1x_2)^t.$$
By Lemma \ref{lem_splitting_criterion_2}, the decomposition $J= I_2^s (x_1x_2)^t +  x_3 I_2^{s+1} I_3^{t-1}$ is a Betti splitting. By Lemma \ref{lem_pd_reg_split},
$$\reg J = \max \{ \reg (I_2^2 f_1^t), \reg (x_3 I_2^{s+1} I_3^{t-1})\} = s + 2t.$$
Hence, $J$ has a linear free resolution.

Now, assume $n\geq 4$ and $t=0$. From the induction hypothesis on $n$, the ideal $J=I_{n-1}^sS$ has a linear free resolution. Assume $t\ge 1$. Using a similar argument as in the case $n=3$, we deduce that the decomposition
\begin{equation}\label{eq_3_1}
    J = I_{n-1}^s I_{n}^t = I_{n-1}^s (f_1)^t +  x_n I_{n-1}^{s+1} I_n^{t-1}
\end{equation}
is a Betti splitting and that $J$ has a linear free resolution. 
\end{proof}
We now compute the Betti numbers of powers of $I_n$. To achieve that for all integers $n,s,t,i$ with $n\geq 2$ we set 
$$e(n,s,t,i)= \begin{cases}
    \beta_i (I_{n-1}^s I_n^t) & \text{ if } s \ge 0, t \ge 0, i \ge 0,\\
    0 & \text{ otherwise}.    
\end{cases}$$

From Eq. \eqref{eq_3_1} and the definition of a Betti splitting it follows that for all $n\geq 3$, $s,i\ge 0$, and $t\ge 1$, we have  
\begin{equation}\label{eq_3_2}
e(n,s,t,i) = e(n,s,0,i)+e(n,s+1,t-1,i)+e(n,s,0,i-1).
\end{equation}
Applying Eq. \eqref{eq_3_2} repeatedly, we deduce that 
\begin{equation}\label{eq_3_3}
    e(n,s,t,i) = e(n,s+t,0,i) + \sum_{\ell = 0}^{t-1} (e(n,s+\ell,0,i) + e(n,s+\ell,0,i-1)).
\end{equation}
Since Betti numbers are preserved under flat base extension, we have
\begin{equation}\label{eq_3_4}
e(n,s,0,i)=e(n-1,0,s,i)\text{ for all }n\geq 3\text{ and }s\geq 1.
\end{equation}
Applying Eq. \eqref{eq_3_3} for $s = 0$ and Eq. \eqref{eq_3_4}, we get 

\begin{align*}
e(n,0,t,i) = e(n-1,0,t,i) + \sum_{\ell = 0}^{t-1} (e(n-1,0,\ell,i) + e(n-1,0,\ell,i-1)),\\
e(n,0,t-1,i) = e(n-1,0,t-1,i) + \sum_{\ell = 0}^{t-2} e(n-1,0,\ell,i) + e(n-1,0,\ell,i-1).
\end{align*}
Hence, for all $n \ge 3$ and all $t,i\ge 0$, we have 
\begin{equation}\label{eq_3_5}
    \begin{split}
        e(n,0,t,i) & =  e(n,0,t-1,i) + e(n-1,0,t,i) - e(n-1,0,t-1,i) \\
        &+ e(n-1,0,t-1,i) + e(n-1,0,t-1,i-1)\\
        & = e(n,0,t-1,i) + e(n-1,0,t,i) + e(n-1,0,t-1,i-1).
    \end{split}
\end{equation}
Let $$\Phi(x,y,z):= \sum_{n\ge 2,t,i\geq 0} e(n,0,t,i) x^{n-2} y^t z^i \in \QQ[[x,y,z]].$$ 
\begin{lem} With the notations above, we have 
    $$\Phi(x,y,z) = \frac{1 + yz}{(1 - y)(1 - x-y- xyz) }.$$
\end{lem}
\begin{proof} Let 
$$\Psi = \sum_{t,i \ge 0} e(2,0,t,i) y^t z^i \in \QQ[[y,z]].$$
By Eq. \eqref{eq_3_5}, we have 
\begin{align*}
    \Phi(x,y,z) & = \Psi + \sum_{n \ge 3,t,i\ge 0} e(n,0,t-1,i) x^{n-2} y^t z^i  \\
    &+ \sum_{n \ge 3,t,i\ge 0} e(n-1,0,t,i) x^{n-2} y^t z^i  + \sum_{n \ge 3,s,i\ge 0} e(n-1,0,t-1,i-1) x^{n-2} y^t z^i\\
    &= \Psi + y(\Phi(x,y,z) - \Psi) + x \Phi(x,y,z) + xyz\Phi(x,y,z).
\end{align*}
Since $I_2=(x_1,x_2)$, we have the obvious identities:
\begin{equation}\label{rec3}
e(2,0,t,i)=\begin{cases} t+1,& i=0 \\t,& i=1 \\0,& \text{otherwise} \end{cases}\text{ for all }t\geq 0.
\end{equation}
Hence, 
$$\Psi = \sum_{t\ge 0} (t+1)y^t + ty^tz = \frac{1 + z}{(1-y)^2} - \frac{z}{1-y}.$$
Thus, we deduce that 
$$\Phi(x,y,z) = \frac{1 + z - z + yz}{(1-y)(1 - y - x - xyz)} = \frac{1 + yz}{(1 - y)(1 - x-y- xyz) }.$$
The conclusion follows.
\end{proof}
We are now ready for the proof of Theorem \ref{teo1}.
\begin{proof}[Proof of Theorem \ref{teo1}] In $\QQ[[x,y,z]]$ we have 
$$\frac{1}{1-y} = \sum_{b \ge 0} y^b, \text{ and } \frac{1}{1 - x -y -xyz} = \sum_{a \ge 0} (x (1+yz) + y)^a.$$
The Betti number $\beta_i(J_{n,n-1}^t)$ which is $e(n,0,t,i)$ is the coefficient $x^{n-2}y^tz^i$ of $\Phi(x,y,z)$. By the identity above, we deduce that it is the same as the coefficient of $y^t z^i$ in 
$$(1 + yz)^{n-1}  \sum_{b\ge 0} y^b \sum_{a \ge n-2} \binom{a}{n-2} y^{a-(n-2)}.$$
which is equal to $\binom{n-1}{i}$ times the coefficient of $y^{t-i}$ in $(\sum_{b \ge 0} y^b) \sum_{a \ge n-2} \binom{a}{n-2} y^{a-n}$ which is equal to 
$$\binom{n-1}{i} \cdot \sum_{a = n-2}^{n+t-i-2} \binom{a}{n-2} = \binom{n-1}{i} \binom{n + s - i -1}{s-i}.$$
The conclusion follows.
\end{proof}

\begin{cor}\label{cor1}
For all $n\geq 3$ and $t \ge 1$, we have
\begin{enumerate}
\item[(1)] $\pd_S(I_n^t ) = \min\{n-1,t\}$,
\item[(2)] $\reg_S(I_n^t ) = (n-1)t$.
\end{enumerate}
\end{cor}

\begin{proof}
(1) From Theorem \ref{teo1}, it follows that 
$$\pd_S(I_n^t)=\max\{i\;:\;\beta_i(I_n^t)=\binom{n-1}{i}\binom{n+t-i-1}{t-i}>0\}=\min\{n-1,t\}.$$
(2) According to Theorem \ref{teo1}, $I_n^t$ has linear free resolution. Hence, its 
regularity is equal to its initial degree, i.e., $\reg(I_n^t)=(n-1)t$, as required.
\end{proof}

\begin{obs} Note that, using the Ausl\"ander-Buchsbaum Theorem, the formula (1) of Corollary \ref{cor1} also follows from \cite[Theorem 3.1]{BC2}.
\end{obs}

\section{Recursive formulae for the Betti numbers of powers of $(n-2)$-path ideals of $n$-cycles}\label{sec_n_2}
Assume that $n \ge 3$ is an integer. In this section, we derive recursive equations for the Betti numbers of the powers of the $(n-2)$-path ideal of the $n$-cycle. We then prove Theorem \ref{teo2}. Recall that $S = \k[x_1,\ldots,x_n]$ is a standard graded polynomial ring over a field $\k$. We fix the following notations throughout the section.
$$f_1 = x_1x_2\cdots x_{n-2}, \ldots, f_n = x_n x_1 \cdots x_{n-3}, I_n = (f_1,\ldots,f_n) \text { and } J_n = (f_1,f_3,\ldots,f_n).$$ 
Note that omitting any $f_j$ from $I_n$, we obtain an ideal isomorphic to $J_n$. Our choice to omit $f_2$ from $I_n$ to get $J_n$ allows us to write $J_n = (f_1) + x_n J_{n-1}$, making the induction arguments easier to grasp. We will prove that there are intertwined relations between the Betti numbers of the following ideals: 
$$B_{n,s,t}:= J_{n}^s I_n^t \text{ and }C_{n,s,t}:=J_n^s (x_1,x_n)^t,$$
where $J_2= I_2 = K[x_1,x_2]$ and $s,t$ are natural numbers. From that, we will deduce our formulae.

First, we decompose $B_{n,s,t} = J_n^s I_n^t$ based on the grading induced by $x_n$.  

\begin{lem}\label{lem_decomposition_for_B_n_s_t} Assume that $n \ge 3$ and $s,t$ are natural numbers. We have 
\begin{align*}
& B_{n,s,t} = K_0 J_{n-1}^0 x_n^0 + \cdots + K_{s+t} J_{n-1}^{s+t} x_n^{s+t},\text{ where }\\
& K_d = \begin{cases}
    f_1^{s-d} (f_1,f_2)^t & \text{ if } d \le \min \{t,s\} \\
    (f_1,f_2)^{t+s-d} & \text{ if } s < d \le t \\
    f_1^{s-d} (f_1,f_2)^t & \text{ if } t < d \le s\\
    (f_1,f_2)^{s+t-d} & \text{ if } \max \{s,t\} < d \le t+s.
\end{cases}
\end{align*}
\end{lem}

\begin{proof} 
Note that $J_n = f_1 + x_n J_{n-1}$ and $I_n = (f_1,f_2) + x_n J_{n-1}$. Hence, we have 
\begin{align*}
    & B_{n,s,t} = J_n^s I_n^t = (f_1 + x_n J_{n-1})^s ((f_1,f_2) + x_n J_{n-1})^t = \\
        & = \left(\sum_{i=0}^s f_1^{s-i} J_{n-1}^i x_n^i\right) \left( \sum_{j = 0}^t (f_1,f_2)^{t-j} J_{n-1}^j x_n^j\right).
\end{align*}
    It follows that for a natural number $d$ with $0 \le  d \le s+t$, we have 
    $$K_d = \sum_i f_1^{s-i} (f_1,f_2)^{t-(d-i)},$$
    where the sum is taken over all non-negative integers $i$ such that $i \le \min(d,s)$ and $d - i \le t$. 
		In other words, the sum is taken over all integers $i$ such that $\max\{0,d-t\} \le i \le \min \{d,s\}$. We first note the following simple formula
  $$\sum_{i=0}^a f_1^i (f_1,f_2)^{a-i} = (f_1,f_2)^a,$$
  for all $a \ge 0$. We now consider several cases to deduce our formulae.
    \begin{enumerate}
    \item[(i)] $d \le \min \{t,s\}$. In this case, we have 
    $$K_d = \sum_{i=0}^d f_1^{s-i} (f_1,f_2)^{t-d + i} = f_1^{s-d} (f_1,f_2)^{t-d} \cdot  \sum_{i=0}^d f_1^{d-i} ( f_1,f_2)^i = f_1^{s-d} (f_1,f_2)^t.$$
    \item[(ii)] $s \le d \le t$. In this case, we have 
    $$K_d = \sum_{i=0}^s f_1^{s-i} (f_1,f_2)^{t-d +i} = (f_1,f_2)^{t-d} 
 \cdot \sum_{i=0}^s f_1^{s-i}(f_1,f_2)^i = (f_1,f_2)^{t+s -d}.$$
    \item[(iii)] $t \le d \le s$. In this case, we have 
    $$K_d = \sum_{i=d-t}^d f_1^{s-i} (f_1,f_2)^{t-d + i} = f_1^{s-d} \cdot \sum_{j=0}^t  f_1^{t-j} (f_1,f_2)^{j} = f_1^{s-d} (f_1,f_2)^t.$$
    \item[(iv)] $\max\{t,s\} \le d \le t+s$. In this case, we have 
    $$K_d = \sum_{i = d-t} ^s f_1^{s-i} (f_1,f_2)^{t-d + i} = \sum_{j=0}^{s+t-d} f_1^{s+t-d-j} (f_1,f_2)^j = (f_1,f_2)^{s+t-d}.$$
    \end{enumerate}
    The conclusion follows.
\end{proof}

We define recursively the ideals $M_j$ by
\begin{equation}\label{eq_M}
    M_{s+t}:= K_{s+t} J_{n-1}^{s+t}\text{ and }M_{j}  := K_j J_{n-1}^j + x_n M_{j+1}\text{ for }0\leq j\leq s+t-1.
\end{equation}
In particular, we have $B_{n,s,t} = M_0 =  K_0  + x_n M_1$.

\begin{lem}\label{lem_intersection_B_n_s_t}
    With the notations above, for all $0 \le j \le s+t - 1$, we have
    $$K_j J_{n-1}^j \cap x_n M_{j+1} = x_n K_j J_{n-1}^j.$$ 
\end{lem}
\begin{proof}
By Lemma \ref{lem_intersection}, it suffices to prove that $K_j  \subseteq K_{j+1} J_{n-1}$ for all $0\leq j\leq s+t-1$. By Lemma \ref{lem_decomposition_for_B_n_s_t}, we need to consider the following cases:
\begin{enumerate}
\item[(i)] $0\le j < \min\{s,t\}$. Since $f_1\in J_{n-1}$ it follows that $f_1^{s-j} \in f_1^{s-j-1} J_{n-1}$. Therefore, 
               $K_j=f_1^{s-j}(f_1,f_2)^t \subset f_1^{s-j-1}(f_1,f_2)^t J_{n-1}=K_{j+1}J_{n-1}$, as required.
\item[(ii)] $s \le j < t$. Since $f_1,f_2 \in J_{n-1}$, it follows that $(f_1,f_2) \subseteq J_{n-1}$. Therefore, $K_j = (f_1,f_2)^{s+t-j} \subseteq (f_1,f_2)^{s+t-j-1} J_{n-1} = K_{j+1} J_{n-1}$.
\item[(iii)] $t \le j < s$. The argument is the same as that of case (i).
\item[(iv)] $j \ge \max\{s,t\}$. The argument is the same as that of case (ii).
\end{enumerate}
The conclusion follows.
\end{proof}

Now, we decompose $C_{n,s,t} = J_n^s (x_1,x_n)^t$, based on the grading induced by $x_n$.

\begin{lem}\label{lem_decomposition_C_n_s_t} Assume that $n \ge 3$ and $s,t$ are natural numbers. Let $L_d$ be the degree $d$-th $x_n$-graded component of $C_{n,s,t}$, i.e., 
$$C_{n,s,t} = L_0 + x_n L_1 + \cdots + x_n^{s+t} L_{s+t}.$$ 
Then
$$L_d = \begin{cases}
        f_1^{s-d} x_1^{t-d} (f_1 + x_1J_{n-1})^d, & \text{ if } d <  \min \{t,s\}\\
        x_1^{t-d} (f_1 + x_1 J_{n-1})^s, & \text{ if } s \le d < t \\
        f_1^{s-d} J_{n-1}^{d-t} (f_1 + x_1J_{n-1})^t, & \text { if } t \le d < s\\
        J_{n-1}^{d-t} (f_1 + x_1 J_{n-1})^{s+t -d}, & \text{ if } \max\{s,t\} < d \le t +s.
\end{cases}$$
\end{lem}

\begin{proof} 
We have $C_{n,s,t} = (f_1 + x_n J_{n-1})^s (x_1,x_n)^t$. Thus, the degree $d$-th component of $C_{n,s,t}$ is 
$$L_d = \sum_i f_1^{s-i} J_{n-1}^i x_1^{t-(d-i)},$$
where the sum is taken over all non-negative integers $i$ such that $i\le \min\{d,s\}$ and $d - i \le t$, i.e., $\max\{0,d -t\} \le i \le \min\{d,s\}$.

We consider the following cases to obtain the desired formulae:
\begin{enumerate}
\item[(i)] $d < \min\{s,t\}$. In this case, we have 
$$L_d = \sum_{i = 0}^d f_1^{s-i} J_{n-1}^i x_1^{t-d + i} = f_1^{s-d} x_1^{t-d} \sum_{i=0}^d f_1^{s-i} (x_1J_{n-1})^i = f_1^{s-d} x_1^{t-d} (f_1 + x_1 J_1)^d.$$
\item[(ii)] $s \le d < t$. In this case, we have 
$$L_d = x_1^{t-d} \sum_{i=0}^s f_1^{s-i} (x_1J_{n-1})^i = x_1^{t-d} (f_1 + x_1 J_{n-1})^s.$$
\item[(iii)] $t \le d \le s$. In this case, we have 
$$L_d = \sum_{i=d-t}^d f_1^{s-i} J_{n-1}^i x_1^{t-d + i}= f_1^{s-d} J_{n-1}^{d-t} \sum_{j=0}^t  f_1^{t - j} (x_1J_{n-1})^j = f_1^{s-d} J_{n-1}^{d-t} (f_1 + x_1 J_{n-1})^t.$$
\item[(iv)] $d \ge \max\{s,t\}$. In this case, we have 
$$L_d = \sum_{i=d-t}^s f_1^{s-i} J_{n-1}^i x_1^{t-d+i} = J_{n-1}^{d-t} \sum_{j=0}^{s+t-d} f_1^{s+t-d-j} (x_1J_{n-1})^j = J_{n-1}^{d-t} (f_1  + x_1J_{n-1})^{s+t-d}.$$
\end{enumerate}
The conclusion follows.
\end{proof}

We define recursively the ideals $N_j$ by 
\begin{equation}\label{eq_N}
    N_{s+t}:=L_{s+t}\text{ and }N_j:= L_j + x_n N_{j+1}\text{ for }0\leq j\leq s+t-1.
\end{equation} 
In particular, we have $C_{n,s,t} = L_0 + x_n N_1$.

\begin{lem}\label{lem_intersection_C_n} 
 With the notations above, we have 
 $$L_i \cap x_n N_{i+1} = x_n L_i\text{ for all }0\leq i\leq s+t-1.$$
\end{lem}

\begin{proof}
    By Lemma \ref{lem_intersection}, it suffices to prove that $L_{j} \subseteq L_{j+1}$ for all $0\leq j \leq s+t-1$. By Lemma \ref{lem_decomposition_C_n_s_t}, we need to consider the following cases. 
    \begin{enumerate}
    \item[(i)] $0 \le j < \min \{s,t\}$. The inclusion is clear from the formula for $L_j$. 
    \item[(ii)] $s \le j < t$. The inclusion follows from the formula for $L_j$. 
    \item[(iii)] $t \le j < s$. The conclusion follows from the fact that $f_1 \in J_{n-1}$ and the formula for $L_j$. 
    \item[(iv)] $j\ge \max\{s,t\}$. The conclusion follows from the fact that $(f_1 + x_1J_{n-1}) \subseteq J_{n-1}$ and the formula for $L_j$. 
\end{enumerate}
The conclusion follows.
\end{proof}

We see that the ideal $J_n^{s}$ is a common ideal that appears in $B_{n,s,t}$ and $C_{n,s,t}$. It is the base case for the induction step. So we treat it first. By the definition of $J_n$, we have $J_n = (f_1) + x_n J_{n-1}$ and $f_1 \in J_{n-1}$. Thus, we have

\begin{lem}\label{lem_A_n_s_t} 
Assume that $n \ge 3$. For all $s \ge 0$ and $t\ge 0$, the ideal $A_{n,s,t} = J_{n-1}^s J_n^t$ has a linear free resolution and, if $t\ge 1$, the decomposition 
$$J_{n-1}^s J_n^t = J_{n-1}^s (f_1^t) + x_n J_{n-1}^{s+1} J_n^{t-1}$$ is a Betti splitting.
\end{lem}
\begin{proof}
The proof is similar to that of Lemma \ref{lem_first_case}.
\end{proof}
Hence, the Betti numbers of $J_{n-1}^s J_n^t$ share the same recursive equations with those of $J_{n-1,n-2}^s J_{n,n-1}^t$. Furthermore, $J_3$ has the same Betti numbers as $J_{2,1}$. We deduce that $J_{n-1}^s J_n^t$ has the same Betti numbers as $J_{n-2,n-3}^s J_{n-1,n-2}^t$. In particular, we have 
\begin{cor}\label{corjn-1}
For all $n\ge 2$ and $t\ge 1$ we have
$$\beta_i(J_{n}^t) = \binom{n-2}{i}\binom{n+t-i-2}{t-i}.$$
In particular, $\pd(J_n^t) = \min(n-2,t)$.
\end{cor}
\begin{proof}
    The conclusion follows from the proof of Theorem \ref{teo1} and Lemma \ref{lem_A_n_s_t}.
\end{proof}

Now comes the technical step toward computing the Betti numbers of $I_n^s$.

\begin{lem}\label{lem_linear_free_resolutions_B_C} For all $n \ge 2$, $s,t\ge 0$, the ideals $B_{n,s,t}$ and $C_{n,s,t}$ have linear free resolutions.     
\end{lem}
\begin{proof}
    We prove by induction on $n$ and then on $t$. The base case $n = 2$ is clear. Now assume that the statement holds for $n-1$. We have $B_{n,s,0} = C_{n,s,0} = A_{n,0,s}$ have linear free resolution by Lemma \ref{lem_A_n_s_t}. Now assume that $t \ge 1$ and the conclusion holds for $t-1$. We will now consider the statement for $B_{n,s,t}$ and $ C_{n,s,t}$ respectively.

    For $B_{n,s,t}$, we now prove by downward induction on $j$ that $M_j$ has a linear free resolution. We have $M_{s+t} = J_{n-1}^{s+t} = A_{n-1,0,s+t}$, which has a linear free resolution by Lemma \ref{lem_A_n_s_t}. By Lemma \ref{lem_decomposition_for_B_n_s_t}, $K_j J_{n-1}^j$ is of the form $f_1^a (f_1,f_2)^b J_{n-1}^j$. Hence, it has a linear free resolution by induction on $n$. By Lemma \ref{lem_intersection_B_n_s_t} and Lemma \ref{lem_splitting_criterion_2}, the decomposition $M_j = K_j J_{n-1}^j + x_n M_{j+1}$ is a Betti splitting. Hence, $M_j$ has a linear free resolution.

    For $C_{n,s,t}$, we prove by downward induction on $j$ that $N_j$ has a linear free resolution. We have $N_{s+t} = I_{n-1}^s = B_{n-1,0,s}$ which has a linear free resolution by induction on $n$. Also, $L_j$ is of the form $f_1^ax_1^b J_{n-1}^u I_{n-1}^v$ for some $a,b,u,v$, hence, has a linear resolution by induction on $n$ as well. By Lemma \ref{lem_intersection_C_n} and Lemma \ref{lem_splitting_criterion_2}, the decomposition $N_j = L_j + x_n N_{j+1}$ is a Betti splitting. Hence, $N_j$ has a linear free resolution. The conclusion follows.    
\end{proof}

For integers $n\ge 2$ and $s,t,i\ge 0$ we set
$$b(n,s,t,i) = \beta_i(B_{n,s,t})\text{ and }c(n,s,t,i) = \beta_i(C_{n,s,t}).$$ 
If $s<0$, $t<0$ or $i<0$ we set $b(n,s,t,i)=c(n,s,t,i)=0$. From Lemmas \ref{lem_decomposition_for_B_n_s_t}, \ref{lem_intersection_B_n_s_t}, \ref{lem_decomposition_C_n_s_t}, \ref{lem_intersection_C_n}, \ref{lem_linear_free_resolutions_B_C} we can deduce formulae for $b$ coefficients in terms of $c$ coefficients and vice versa. For convenience in writing these formulae, we set 
$$\tilde b(n,s,t,i) = b(n,s,t,i) + b(n,s,t,i-1)\text{ and }\tilde c(n,s,t,i) = c(n,s,t,i) + c(n,s,t,i-1).$$

\begin{lem}\label{lem_b_c} Assume that $n \ge 3$ and $s,t,i\geq 0$. Let 
$$\Lambda(s,t) = \{ (j,t)\;:\;0\leq j\leq s\}\cup \{(s+j,t-j)\;:\;1\leq j\leq t-1\}.$$ 
We have
$$b(n,s,t,i) = b(n-1,s+t,0,i) + \sum_{(a,b) \in \Lambda(s,t)} \tilde c(n-1,a,b,i).$$
\end{lem}

\begin{proof} 
We recall Eq. \eqref{eq_M}
$$
    M_{s+t} = K_{s+t} J_{n-1}^{s+t} = J_{n-1}^{s+t} \text{ and }
    M_j  = K_j J_{n-1}^j + x_n M_{j+1}.
$$
First, we have 
\begin{equation}\label{eq_ind_b}
    \beta_i(M_{s+t}) = \beta_i(J_{n-1}^{s+t}) = c(n-1,s+t,0,i) = b(n-1,s+t,0,i).
\end{equation} 
Let $\lambda_j(s,t) = \{ (a,b) \in \Lambda(s,t) \mid a \ge j\}$. We prove by downward induction on $j$ that 
$$\beta_i(M_j) = b(n-1,s+t,0,i) + \sum_{(a,b) \in \Lambda_j } \tilde c(n-1,a,b,i).$$
The base case $j = s+t$ follows from Eq. \eqref{eq_ind_b}. Now assume that $0\le j < s+t$. By the proof of Lemma \ref{lem_linear_free_resolutions_B_C}, the decomposition $$M_j = K_j J_{n-1}^j + x_n M_{j+1}$$ is a Betti splitting. Furthermore, by Lemma \ref{lem_intersection_B_n_s_t}, 
$$K_j J_{n-1}^j \cap x_n M_{j+1} = x_n K_j J_{n-1}^j.$$ 
Hence, by the definition of Betti splitting, we have
$$\beta_i(M_j) = \beta_i(M_{j+1}) + \beta_i(K_j J_{n-1}^j) + \beta_{i-1} (K_j J_{n-1}^j).$$
The conclusion follows from Lemma \ref{lem_decomposition_for_B_n_s_t}.    
\end{proof}

\begin{lem}\label{lem_c_b}  Assume that $n\ge 3$ and $s,t,i \ge 0$. We define the multiset $\Delta(s,t)$ as follows. If $s \le t$ then 
$$\Delta(s,t) = \{(0,0), \ldots,  (0,s)^{(t-s+1)},(1,s-1), \ldots, (s-1,1)\},$$ where the notation $(0,s)^{(t-s+1)}$ means that it appears $t-s+1$ times in $\Delta(s,t)$. If $t < s$ then 
$$\Delta(s,t) = \{(0,0), \ldots,  (0,t), (1,t),\ldots, (s-t,t), (s-t+1,t-1), \ldots, (s-1,1)\}.$$
We have 
$$ c(n,s,t,i) =  b(n-1,s,0,i) + \sum_{(a,b) \in \Delta(s,t)} \tilde b(n-1,a,b,i).$$
\end{lem}

\begin{proof}
We recall Eq. \eqref{eq_N}
$$
    N_{s+t} = J_{n-1}^s \text{ and }
    N_j  = L_j + x_n N_{j+1}.
$$
First, we have 
\begin{equation}\label{eq_ind_c}
    \beta_i(N_{s+t}) = \beta_i(J_{n-1}^s) = b(n-1,s,0,i).
\end{equation}
Let $\Delta_j$ be the (multi)set consisting of the last $j$ elements of $\Delta(s,t)$. We prove by downward induction on $j$ that 
$$\beta_i(N_j) = b(n-1,s,0,i) + \sum_{(a,b) \in \Delta_j(s,t)} b(n-1,a,b,i,i-1).$$ 
The base case $j = s+t$ follows from Eq. \eqref{eq_ind_c}. By the proof of the Lemma \ref{lem_linear_free_resolutions_B_C}, the decomposition 
$$N_j = L_j + x_n N_{j+1}$$
is a Betti splitting. Furthermore, by Lemma \ref{lem_intersection_C_n}, 
$$L_j \cap x_n N_{j+1} = x_n L_j.$$
Hence, by the definition of Betti splittings, we have 
$$ \beta_i(N_j) = \beta_i (N_{j+1}) + \beta_i (L_j) + \beta_{i-1} (L_j).$$
The conclusion follows from Lemma \ref{lem_decomposition_C_n_s_t}.
\end{proof}

From Lemma \ref{lem_b_c} and Lemma \ref{lem_c_b}, we can derive the self-recurrent relation among the $b$ coefficients. For convenience, for all natural number $n \ge 2$ and all integers $u,v,i$, we set 
$$\dbtilde{b}(n,u,v,i) = b(n,u,v,i) + 2b(n,u,v,i-1) + b(n,u,v,i-2).$$
First, we give the form of the self-recurrent relation.

\begin{lem}\label{lem_support_non_zero_coeffs} Assume that $n \ge 4$ and $s,t,i \ge 0$. Let $r = \lfloor \frac{s + t}{2} \rfloor$ and 
\begin{equation*}
    \Gamma(s,t):= \{(u,v) \in \NN^2 \mid 0 < v \le \min(r,t), u + 2v \le s+t\}\} \cup \{(0,0)\}.
\end{equation*}
We have
\begin{align*}
    b(n,s,t,i) &= b(n-1,s+t,0,i) + \sum_{u=0}^{s+t-1} \tilde b(n-2,u,0,i) \\
    &+ \sum_{(u,v) \in \Gamma(s,t) } f(u,v) \dbtilde{b}(n-2,u,v,i)
\end{align*}
where $f(u,v)$'s are some non-negative integers.
\end{lem}

\begin{proof}
From Lemma \ref{lem_b_c} it follows that
\begin{align*}
b(n,s,t,i) & = b(n-1,s+t,0,i) + \tilde c(n-1,0,t,i) + \cdots + \tilde c(n-1,s,t,i) \\
& + \tilde c(n-1,s+1,t-1,i) + \cdots + \tilde c(n-1,s+t-1,1,i).
\end{align*}
By Lemma \ref{lem_c_b} each term $\tilde c(n-1,a,b,i)$ gives rise to a term $\tilde b(n-2,a,0,i)$ and the terms $\tilde b(n-2,u,v,i)$ for $(u,v) \in \Delta(a,b)$. Let 
$$\Omega = \bigcup_{(a,b) \in \in \Lambda(s,t)} \Delta(a,b).$$
Then, we have 
\begin{align*}
    b(n,s,t,i) &= b(n-1,s+t,0,i) + \sum_{u=0}^{s+t-1} \tilde b(n-2,u,0,i) \\
    &+ \sum_{(u,v) \in \Omega } f(u,v) \dbtilde{b}(n-2,u,v,i)
\end{align*}
for some non-negative integers $f(u,v)$. It remains to show that $\Delta(a,b) \subseteq \Gamma(s,t)$ for all $(a,b) \in \Lambda(s,t)$. 
\begin{enumerate}
\item[(i)] Assume that $ (a,b) = (j,t)$ for some $j \le \min (r,s)$. Then  
$$\Delta(j,t) = \{ (u,v) \in \NN^2 \mid 1 \le  v \le j-1, u + v = j \} \cup \{(0,0),\ldots,(0,j)\} \subseteq \Gamma(s,t).$$
\item[(ii)] Assume that $(a,b) = (j,t)$ for some $j$ such that $\min(r,s) < j \le s$. Then 
$$\Delta(j,t) = \{(0,0), \ldots,(0,t),(1,t), \ldots,(s-t,t),(s-t+1,t-1), \ldots,(s-1,1)\} \subseteq \Gamma(s,t).$$
\item[(iii)] Assume that $(a,b) = (j,s+t - j)$ for some $j > s$. If $j \le s+t - j$, the argument is similar to case (i). If $j > s+t-j$ then 
\begin{align*}
    \Delta(j,s+t-j) &= \{(0,0), \ldots,(0,s+t-j), (1,s+t-j), \ldots,(2j-s-t,s+t-j)\}\\
    &\cup \{(2j-s-t+1,s+t-j-1), \ldots,(j-1,1)\}\subset \Gamma(s,t).
\end{align*}
\end{enumerate}
That concludes the proof of the lemma.
\end{proof}

The following pictures represent the sets $\Lambda(s,t)$, $\Delta(u,v)$ and $\Gamma(s,t)$. To get $\Gamma(s,t)$, first we draw the blue line corresponding to $\Lambda(s,t)$. Then, for each point $(u,v)$ on the blue line, we draw red lines corresponding to $\Delta(u,v)$. $\Gamma(s,t)$ consists of all the integer points on the red lines.

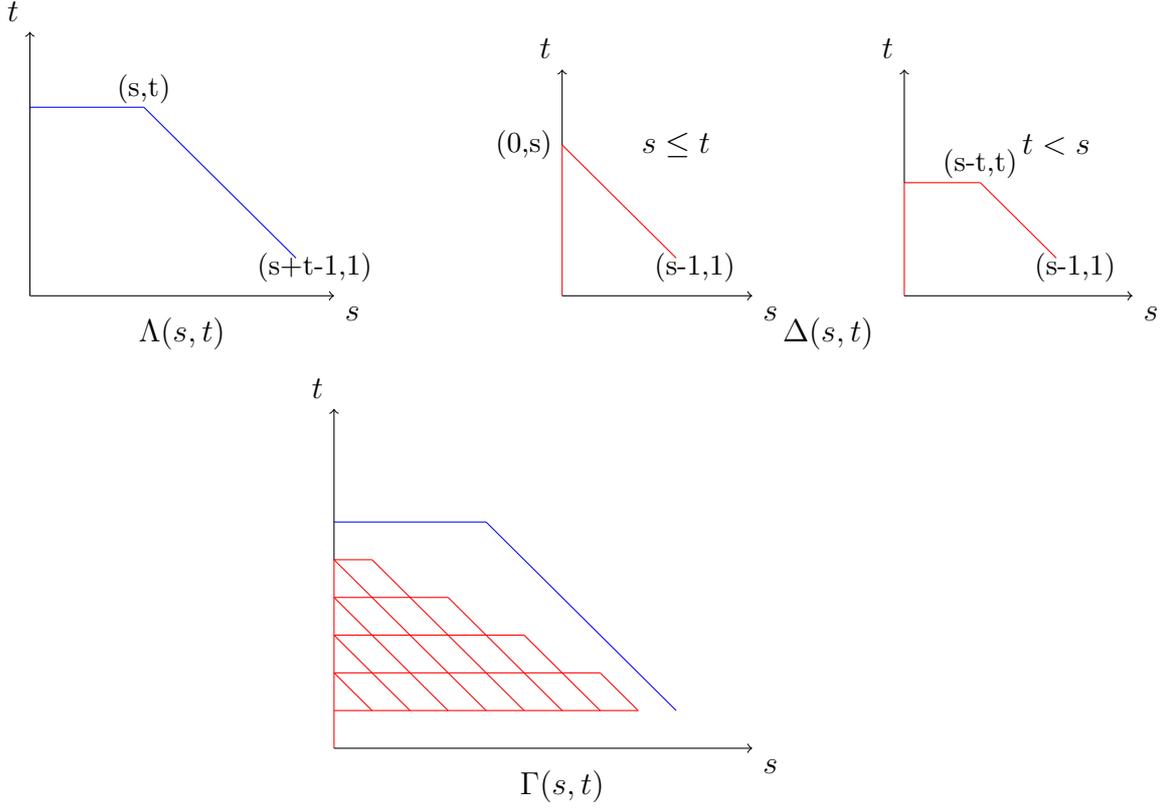
\begin{center}
\begin{tikzpicture}[scale=0.50]
\draw[blue] (0,5) -- (3,5);
\draw[blue](3,5) -- (7,1);
\draw[->] (0,0) -- (8,0) node[anchor=north west] {$s$};
\draw[->] (0,0) -- (0,7) node[anchor=south east] {$t$};
\node at (3,5.5) {\small{(s,t)}};
\node at (7.5,0.75) {\small{(s+t-1,1)}};
\node at (4,-1) {$\Lambda(s,t)$};

\draw[red] (14,0) -- (14,4);
\draw[red](14,4) -- (17,1);
\draw[->] (14,0) -- (19,0) node[anchor=north west] {$s$};
\draw[->] (14,4) -- (14,6) node[anchor=south east] {$t$};
\node at (17,4) {$s\le t$};
\node at (13,4) {\small{(0,s)}};
\node at (17.5,0.75) {\small{(s-1,1)}};

\draw[red] (23,0) -- (23,3);
\draw[red](23,3) -- (25,3);
\draw[red](25,3) -- (27,1);
\draw[->] (23,0) -- (29,0) node[anchor=north west] {$s$};
\draw[->] (23,3) -- (23,6) node[anchor=south east] {$t$};
\node at (27,4) {$t<s$};

\node at (21,-1) {$\Delta(s,t)$};
\node at (25,3.5) {\small{(s-t,t)}};
\node at (27.5,0.75) {\small{(s-1,1)}};

\draw[blue] (8,-6) -- (12,-6);
\draw[blue](12,-6) -- (17,-11);
\draw[red](8,-12) -- (8,-7);
\draw[red](8,-11) -- (16,-11);
\draw[red](8,-10) -- (9,-11);
\draw[red](8,-9) -- (10,-11);
\draw[red](8,-8) -- (11,-11);
\draw[red](8,-7) -- (12,-11);
\draw[red](8,-7) -- (9,-7);
\draw[red](9,-7) -- (13,-11);
\draw[red](8,-8) -- (11,-8);
\draw[red](11,-8) -- (14,-11);
\draw[red](8,-9) -- (13,-9);
\draw[red](13,-9) -- (15,-11);
\draw[red](8,-10) -- (15,-10);
\draw[red](15,-10) -- (16,-11);
\draw[->] (8,-12) -- (19,-12) node[anchor=north west] {$s$};
\draw[->] (8,-7) -- (8,-3) node[anchor=south east] {$t$};

\node at (14,-13) {$\Gamma(s,t)$};
\end{tikzpicture}
\end{center}

We then deduce:
\begin{cor}\label{cor_pd} Assume that $n \ge 4$ and $s,t \ge 0$. Let $q = \max\{ \pd (J_{n-2}^u I_n^v) \mid (u,v) \in \Gamma(s,t) \text{ and } f(u,v) \neq 0\}$. Then 
$$\pd (J_n^s I_n^t) = \max \{ \min\{n-3,s+t\}, q+2\}.$$     
\end{cor}
\begin{proof} Let 
$p = \max \{\pd (J_{n-2}^u) \mid u = 0, \ldots, s+t-1\}$. 
By the definition of projective dimension, we have 

\begin{align*}
    \pd(J_n^s I_n^t) = &\max \{i \mid b(n,s,t,i) \neq 0\} \\
    = &\max \{ i \mid b(n-1,s+t,0,i) \neq 0 \\
    &\text{ or } \tilde b(n-2,u,0,i) \neq 0 \text{ for some } u = 0, \ldots, s+t - 1 \\
    & \text{ or } \dbtilde{b}(n-2,u,v,i) \neq 0 \text{ for some} (u,v) \in \Gamma(s,t) \text{ and } f(u,v) \neq 0\}.
\end{align*}
Since $\max\{a \in A, b\in B \} = \max \{\max\{a \in A\} , \max \{b \in B\} \} $, we deduce that 
$$\pd (J_n^s I_n^t) = \max \{\pd (J_{n-1}^{s+t}), p + 1, q+ 2\}.$$
By Corollary \ref{corjn-1}, we have $\pd (J_{n-1}^{s+t}) = \min(n-3,s+t)$ and $p = \min(n-4,s+t-1)$. The conclusion follows.
\end{proof}

We now prove one of the main results of the section about the projective dimension of $J_n^sI_n^t$.
\begin{thm}\label{thm_pd_n-2}
    Assume that $t \ge 1$ and $n \ge 2$. Then $\reg (J_n^sI_n^t) = (t+s)(n-2)$ and 
    $$\pd(J_n^sI_n^t) = \begin{cases}
        \min (n-1, 2(s+t)) & \text{ if } n \text{ is odd } \\
        \min (n-2,2(s+t)) & \text{ if } n \text{ is even}.
    \end{cases}$$
\end{thm}
\begin{proof}
By Lemma \ref{lem_linear_free_resolutions_B_C}, $J_n^sI_n^t$ has a linear free resolution, hence $\reg (J_n^s I_n^t) = (t+s) ( n-2)$. 
Now we prove the formula for the projective dimension. For ease of reading, we divide the proof into several steps.

\vspace{1.5mm}
\noindent {\textbf{Step 1.}} With the notation as in Lemma \ref{lem_support_non_zero_coeffs}, $f(s+t-2,1) = 1$. Indeed, we have $(s+t-1,1) \in \Lambda(s,t)$ and $(s+t-2,1) \in \Delta(s+t-1,1)$. Hence, $f(s+t-2,1) \ge 1$. Furthermore, for any $(a,b) \in \Lambda(s,t)$ with $b > 1$, we have $a \le s+t - 2$. In particular, for any $(u,v) \in \Delta(a,b)$ we have $u+v \le a = s+t - 2$. Thus, $(s+t-2,1) \notin \Delta(a,b)$. Therefore, $f(s+t-2,1) = 1$.

\vspace{1.5mm}
\noindent{\textbf{Step 2.}} The base case $n = 2$ and $n = 3$. If $n = 2$, then $J_n^sI_n^t = S$; hence $\pd(J_n^sI_n^t) = 0$. Assume that $n = 3$. If $s = 0$ and $t = 1$ then the conclusion is clear. Thus, we may assume that $s + t \ge 2$. By Lemma \ref{lem_b_c}, we have 
$$b(3,s,t,i) = b(2,s+t,0,i) + \sum_{(a,b) \in \Lambda(s,t)} \tilde c(2,a,b,i).$$
By the definition of projective dimension, we have 
\begin{align*}
    \pd(J_3^s I_3^t) &= \max \{i \mid b(3,s,t,i) \neq 0\} \\
    &= \max \{ i \mid b(2,s+t,0,i) \neq 0 \text{ or } \tilde c(2,a,b,i) \neq 0 \text{ for some } (a,b) \in \Lambda(s,t)\} \\
    &= \max \{\pd (J_2^{s+t}), \pd (J_2^a (x_1,x_2)^b) + 1 \mid (a,b) \in \Lambda(s,t)\}.
\end{align*}
Now, for any $b \ge 1$ and $a \ge 0$ we have $J_2^a (x_1,x_2)^b = (x_1,x_2)^b$ has projective dimension $1$. Since $t \ge 1$ and $s + t \ge 2$ there exists $(a,b) \in \Lambda(s,t)$ with $b \ge 1$. Hence, $\pd (J_n^s I_n^t) = 2$. 

\vspace{1.5mm}
\noindent{\textbf{Step 3.}} We now prove by induction on $n$ then on $s$ that 
$$\pd(J_n^sI_n) = \begin{cases}
        \min (n-1, 2(s+1)) & \text{ if } n \text{ is odd } \\
        \min (n-2,2(s+1)) & \text{ if } n \text{ is even}.
    \end{cases}$$
By Step 2, we may assume that $n \ge 4$. The base case $s = 0$ follows from the result of Alilooee and Faridi \cite{AF}. Now assume that $s \ge 1$. We have $\Gamma(s,1) = \{ (0,0), (0,1), \ldots, (s-1,1)\}$. Let 
$$q:= \max \{ \pd (J_{n-2}^u I_{n-2} \mid u = 0,\ldots s-1\}.$$
By induction on $n$, we have 
\begin{equation}\label{eq_q_term}
    q =  \begin{cases}
        \min (n-3, 2s) & \text{ if } n \text{ is odd } \\
        \min (n-4,2s) & \text{ if } n \text{ is even}.
    \end{cases}.
\end{equation}
By Corollary \ref{cor_pd}, we have 
$$\pd (J_n^s I_n) = \max \{ \min \{ n-3, s+1\}, q+2\} = q + 2.$$
By Eq. \eqref{eq_q_term}, Step 3 follows.

\vspace{1.5mm}
\noindent{\textbf{Step 4.}} General case. Now assume that $t \ge 2$. By Corollary \ref{cor_pd}, we have $\pd (J_n^s I_n^t) = \max \{ \min \{ n-3, s+t\} , q+ 2\}$, where 
$$q = \max \{\pd  (J_{n-2}^u I_{n-2}^v) \mid (u,v) \in \Gamma(s,t) \text{ and } f(u,v) \neq 0\}.$$ 
By induction, Step 1, and Step 3, we deduce that 
$$q = \pd (J_{n-2}^{s+t-2} I_{n-2}) = \begin{cases}
        \min (n-3, 2(s+t-1)) & \text{ if } n \text{ is odd } \\
        \min (n-4,2 (s+t-1)) & \text{ if } n \text{ is even}.
    \end{cases}.$$
The conclusion follows.
\end{proof}

Now, we can prove Theorem \ref{teo2}:
\begin{proof}[Proof of Theorem \ref{teo2}]
    The conclusion follows by taking $s = 0$ in Theorem \ref{thm_pd_n-2}.
\end{proof}

\begin{obs} From Lemma \ref{lem_b_c} and Lemma \ref{lem_c_b}, we can give formulae for $f(u,v)$ for all $(u,v) \in \Gamma(s,t)$. Nonetheless, they are hard to use to derive the formulae for the Betti numbers of powers of $I_n$. In the following section, we will derive self-recurrent equations for the $b$ coefficients, which are easier to use.    
\end{obs}

\section{Betti numbers of powers of $(n-2)$-path ideals of $n$-cycles}\label{sec_experiment} 
In this section, we give the formulae for the Betti numbers of powers of $n-2$-path ideals of $n$-cycles. First, we recall several notations from the previous section. Let $S_n = \k[x_1,\ldots,x_n]$ be the standard graded polynomial ring over a field $\k$ where $n \ge 2$ is an integer. By convention, we set $J_2 = I_2 = S_2$. Now, assume that $n \ge 3$. Let $f_1=x_1x_2\cdots x_{n-2},\ldots, f_n=x_nx_1\cdots x_{n-3}$.
Let $I_n=(f_1,f_2,\ldots,f_n)$ and $J_n=(f_1,f_3,\ldots,f_n)$. For $s,t\ge 0$, we set 
$$B_{n,s,t}=J_n^sI_n^t \text{ and } C_{n,s,t} = J_n^s (x_1,x_n)^t.$$
Furthermore, if $s,t,i \ge 0$, we let $b(n,s,t,i) = \beta_i(B_{n,s,t})$ and $c(n,s,t,i) = \beta_i(C_{n,s,t})$. If one of $s,t,i$ is negative, by convention, we set $b(n,s,t, i) = c(n,s,t, i) = 0$.

We now introduce the following functions in order to give the formula for $b(n,s,t, i)$. 
\begin{align*}
    p(n,s,t,i)^+ &= \sum_{j=0}^{\left\lfloor \frac{i}{2} \right\rfloor} \binom{n}{i-2j} \left [ \binom{n+s+t-1-i+j}{n-1}-
 \binom{n+s-1-i+j}{n-1}\right ],\\
 p(n,s,t,i)^-&= \begin{cases}
     \sum_{j=0}^{\left\lfloor \frac{i-1}{2} \right\rfloor} \binom{n}{i-1-2j}\left(\binom{s+t+k-1-j}{n-1}-\binom{s+k-1-j}{n-1}\right) & \text{ if } n \text{ is odd}, \\
     \sum_{j=0}^{\left\lfloor \frac{i}{2} \right\rfloor} \binom{n}{i-2j}\left(\binom{s+t+k-1-j}{n-1}-\binom{s+k-1-j}{n-1}\right) & \text{ if } n \text{ is even},
 \end{cases}\\
 p(n,s,t,i)^c &= \binom{n-2}{i}\binom{n+s-i-2}{n-2},\\
 p(n,s,t,i) &= p(n,s,t,i)^+ - p(n,s,t,i)^- + p(n,s,t,i)^c.
\end{align*}
The $+$, $-$, and $c$ stand for the positive part, the negative part, and the constant part of $p$. The entire section is devoted to proving the following main technical result.
\begin{lem}\label{lem_main}
    For all integers $n,s,t,i$ with $n \ge 2$, $s,t\ge 0$, we have 
    $$b(n,s,t,i) = p(n,s,t,i).$$
\end{lem}

We now outline the strategy of proving the lemma. The tuples $(n,s,t, i)$ are totally ordered by the lexicographical order with $n > t > s > i$. The tuples $(2,s,t,i)$, $(3,s,t,i)$, and $(n,s,0,i)$ are called the boundary tuples. To establish Lemma \ref{lem_main}, we will prove the following

\begin{enumerate}
    \item $b(n,s,t,i) = p(n,s,t,i)$ for all boundary tuples $(n,s,t,i)$.
    \item $b(n,s,t, i)$ and $p(n,s,t, i)$ satisfy the same recurrent relation expressing their value at a non-boundary tuple as a finite linear combination of their values at smaller tuples and boundary tuples.
\end{enumerate}
For convenience in writing the recurrent relation, we inherit the following convention from the previous section. For any function $f$ in the tuples $(n,s,t,i)$ we set 
\begin{align*}
    \tilde f(n,s,t,i) &= f(n,s,t,i) + f(n,s,t,i-1),\\
    \dbtilde{f} (n,s,t,i) &= f(n,s,t,i) + 2 f(n,s,t,i-1) + f(n,s,t,i-2).
\end{align*}

First, we prove that the functions $b$ and $p$ agree on the boundary tuples.

\begin{lem}\label{lem_init_2_s_t}
    For all integers $s,t,i$ with $t \ge 0$ we have that $b(2,s,t,i) = p(2,s,t,i)$.
\end{lem}
\begin{proof}

Since, $J_2 = I_2 = S_2$, we have 
$$b(2,s,t,i) = \begin{cases} 1 & \text{ if } i =0 \text{ and } s, t \ge 0,\\
    0 & \text{ otherwise}.\end{cases}
$$
Now, by the definition of $p(n,s,t,i)$, we have
   $$ p(2,s,t,i) =  \binom{0}{i} \binom{s-i}{0}  = \begin{cases}
       1 & \text{ if } i = 0 \text{ and } s \ge 0\\
       0 & \text{ otherwise}.
   \end{cases}
       $$
The conclusion follows.       
\end{proof}

\begin{lem}\label{lem_init_n_s_0}
    For all $n \ge 2$ and all $s,i$ we have $b(n,s,0,i) = p(n,s,0,i)$.
\end{lem}
\begin{proof} The case where $n = 2$ follows from Lemma \ref{lem_init_2_s_t}. The case where $n \ge 3$ follows from the definition of $p(n,s,t,i)$ and Corollary \ref{corjn-1}.   
\end{proof}

\begin{lem}\label{lem_init_3_s_t_i} For all integers $s,t,i$ with $s,t\ge 0$ we have that $b(3,s,t,i) = p(3,s,t,i)$.
\end{lem} 

\begin{proof} By Lemma \ref{lem_init_n_s_0} and Corollary \ref{corjn-1}, we have 
\begin{equation}\label{eq_b_3_0}
    p(3,s,0,i) = b(3,s,0,i) = \binom{1}{i} (s - i + 1) = \begin{cases} s + 1 & \text{ if } i = 0 \\
s & \text{ if } i = 1 \\
0 & \text{ if } i \ge 2.\end{cases}
\end{equation}
Hence, we may assume that $t \ge 1$. First, we calculate $b(3,s,t,i)$. By Lemma \ref{lem_b_c}, we have 
$$b(3,s,t,i) = b(2,s+t,0,i) + \sum_{(a,b) \in \Lambda(s,t)} \tilde c(2,a,b,i),$$
where 
$$\Lambda(s,t) = \{ (j,t)\;:\;0\leq j\leq s\}\cup \{(s+j,t-j)\;:\;1\leq j\leq t-1\}.$$
Since $J_2 = S_2$, we have $c(2,a,b,i) = 0$ if $i \ge 2$ and 
$$c(2,a,b,i) = \begin{cases}
    b+1 & \text{ if } i =0 \\
    b & \text{ if } i =1.
\end{cases}$$
Hence, $\tilde c(2,a,b,i) = 0$ if $i \ge 3$ and 
\begin{equation}\label{eq_c_2}
    \tilde c(2,a,b,i) = \begin{cases}
    b + 1 & \text{ if } i =0 \\
    2b + 1 & \text{ if } i = 1\\
    b & \text{ if } i = 2.
\end{cases}
\end{equation}
Thus, we have $b(3,s,t,i) = 0$ if $i \ge 3$. It remains to consider the cases where $0 \le i \le 2$. From Eq. \eqref{eq_b_3_0} and Eq. \eqref{eq_c_2}, we deduce the following
\begin{enumerate}
\item[(i)] $b(3,s,t,0) = 1 + (s+1) (t+1) + \sum_{b = 1}^{t-1} (b+1) = (t+1)(s+1 + \frac{t}{2}).$
\item[(ii)] $b(3,s,t,1) = (s+1)(2t + 1)  + \sum_{b=1}^{t-1} (2b + 1) =  2st + 2t + s + t^2.$
\item[(iii)] $b(3,s,t,2) = (s+1) t + \sum_{b=1}^{t-1}b = st + \frac{t(t+1)}{2}.$
\end{enumerate}

We now calculate $p(3,s,t,i)$. First, assume that $i = 2 \ell + 1$ for some $\ell \ge 1$. Then we have 

\begin{align*}
    \binom{3}{1}  \left [ \binom{s+t -\ell +1}{2} - \binom{s-\ell +1}{2} \right ]& = \binom{3}{2}  \left [ \binom{s+t -\ell +1}{2} - \binom{s-\ell +1}{2} \right ]\\
    \binom{3}{3}  \left [ \binom{s+t -\ell}{2} - \binom{s-\ell}{2} \right ]& = \binom{3}{0}  \left [ \binom{s+t -\ell}{2} - \binom{s-\ell}{2} \right ]
\end{align*}
and all other terms are $0$. Hence, $p(3,s,t,i) = 0$. The case where $i = 2 \ell$ for some $\ell \ge 2$ can be done similarly. Hence, it remains to consider the cases where $0\le i \le 2$. We have 
\begin{align*}
    p(3,s,t,0) &= \binom{2+s+t}{2} - \binom{2+s}{2} = \frac{t(t+1)}{2} + (s +1)t + (s+1)\\
 p(3,s,t,1) & = 3 \left ( \frac{t(t+1)}{2} + s t \right )- \left ( \frac{t(t+1)}{2} + (s-1) t \right ) + s \\
    & = t(t+1) + 2st + s + t\\
p(3,s,t,2) & = 2 t(t+1) + ( 4 (s - 1) + 1) t  - 3 \left ( \frac{t(t+1)}{2} + (s-1) t \right )  \\
    & = \frac{t(t+1)}{2} + st.
\end{align*}
The conclusion follows.    
\end{proof}

Now, we give recurrent relation for $b(n,s,t,i)$ when $(n,s,t,i)$ is not a boundary tuple. First, we consider the next to the boundary case, namely, the case $t = 1$.
\begin{lem}\label{lem_self_rec_1} Assume that $n \ge 4$, $s \ge 0$ are natural numbers. Then 
\begin{align*}
    b(n,s+1,1,i) & = b(n,s,1,i) + b(n-1,s+2,0,i) - b(n-1,s+1,0,i) \\
    & + \tilde b(n-2,s+1,0,i) + \sum_{j=0}^s \dbtilde{b} (n-2,j,1,i) + \dbtilde{b}(n-2,0,0,i).
\end{align*}
\end{lem}
\begin{proof}By Lemma \ref{lem_b_c}, we deduce that 
$$b(n,s+1,1,i) - b(n,s,1,i) - (b(n-1,s+2,0,i) - b(n-1,s+1,0,i)) = \tilde c(n-1,s+1,1,i).$$
The conclusion then follows from Lemma \ref{lem_c_b}.    
\end{proof}
\begin{lem}\label{lem_self_rec_2} Assume that $n \ge 4$, $s \ge 0$, and $t \ge 2$ are natural number. Then $b(n,s,t,i) - b(n,s+1,t-1,i) $ is equal to 

\begin{alignat*}{2}
  &\sum_{j = 0}^s \dbtilde{b} (n-2,0,j,i) && \text{ if } s \le t  \\
  &\sum_{j=0}^{t-1} \dbtilde{b} (n-2,0,j,i) + (s-t+1)\dbtilde{b}(n-2,0,t,i) && \\
  &+ \sum_{\ell = 1}^{s-t} (s-t+1-\ell) \left (\dbtilde{b}(n-2,\ell,t,i) -\dbtilde{b}(n-2,\ell,t-1,i) \right )&&\text{ if } s \ge t.
\end{alignat*}

\end{lem}
\begin{proof} By Lemma \ref{lem_b_c} we deduce that 
\begin{equation}\label{eq_rec_s_t_0}
    b(n,s,t,i) - b(n,s+1,t-1,i) = \sum_{j = 0}^s (\tilde c(n-1,j,t,i) - \tilde c(n-1,j,t-1,i)).
\end{equation}
Since $t \ge 2$, by Lemma \ref{lem_c_b} we have  $c(n-1,j,t,i) - c(n-1,j,t-1,i)$ is equal to
\begin{align*}
    \tilde b(n-2,0,j,i) & \text{ if } j \le t-1.\\
    \tilde b(n-2,0,t,i) + \sum_{\ell = 1}^{j-t} (\tilde b(n-2,\ell,t,i) - \tilde b(n-2,\ell,t-1,i)) & \text{ if } j > t-1.
\end{align*}
The conclusion follows.    
\end{proof}

It remains to prove that $p(n,s,t,i)$ follows the same recurrent relation as $b(n,s,t,i)$. Since $p(n,s,t,i) = p(n,s,t,i)^+ - p(n,s,t,i)^- + p(n,s,t,i)^c$, we make the following convention. When we have $A = \sum_{\mathbf{t}} a_\mathbf{t} p_\mathbf{t}$ where the sum is over some finite set of tuples $\mathbf{t}$, then $A^+ = \sum_\mathbf{t} a_{\mathbf{t}} p_{\mathbf{t}}^+$ and so on.

Before proving the recurrent relation for $p(n,s,t,i)$ we collect some useful binomial identities.

\begin{lem}\label{lem_binomial_identity} Assume that $n,m$ and $s$ are natural numbers. Then we have 
\begin{enumerate}
    \item $\binom{n+1}{s+1} = \binom{n}{s} + \binom{n}{s+1}$.
    \item $\binom{n-2}{m-2} + 2\binom{n-2}{m-1} + \binom{n-2}{m} = \binom{n}{m}$.
    \item $\sum_{j=0}^s \binom{n+j}{m} = \binom{n+s+1}{m+1} - \binom{n}{m+1}.$
    \item $\sum_{j=0}^s j \binom{n+j}{m} = s \binom{n+s+1}{m +1} - \binom{n+s+1}{m+2} + \binom{n+1}{m+2}$.
\end{enumerate}
\end{lem}
\begin{proof}
    These are standard binomial identities, see e.g. \cite{G}.
\end{proof}

\begin{lem}\label{lem_p_rec_1} Assume that $n \ge 4$, $s \ge 0$ are natural numbers. Then 
\begin{align*}
    p(n,s+1,1,i) & = p(n,s,1,i) + p(n-1,s+2,0,i) - p(n-1,s+1,0,i) \\
    & + \tilde p(n-2,s+1,0,i) + \sum_{\ell=0}^s \dbtilde{p} (n-2,\ell,1,i) + \dbtilde{p}(n-2,0,0,i).
\end{align*}   
\end{lem}
\begin{proof}Let $A(i) = p(n,s+1,1,i) - p(n,s,1,i)$. By Lemma \ref{lem_binomial_identity} and the definition of $p(n,s,t,i)$ we have that 
\begin{align*}
    A(i)^+& =\sum_{j=0}^{\left\lfloor \frac{i}{2} \right\rfloor} \binom{n}{i-2j} \binom{n+s-1-i+j}{n-3}  \\
A(i)^-& = \sum_{j=0}^{\left\lfloor \frac{i-1}{2} \right\rfloor} \binom{n}{i-1-2j} \binom{s+k-1-j}{n-3} \\
A(i)^c &= \binom{n-2}{i} \binom{n+s-i-2}{n-3} 
\end{align*} 
Let $B(i) = \sum_{\ell = 0}^s p(n-2,\ell,1,i) + p(n-2,0,0,i)$. Then we have that 
\begin{align*}
B(i)^+ &=\sum_{j=0}^{\left\lfloor \frac{i}{2} \right\rfloor} \binom{n-2}{i-2j} \left ( \binom{n+s-i+j-2}{n-3} - \binom{n-i+j-3}{n-3} \right ) + p(n-2,0,0,i) \\
&= \sum_{j=0}^{\left\lfloor \frac{i}{2} \right\rfloor} \binom{n-2}{i-2j}  \binom{n+s-i+j-2}{n-3}\\
B(i)^-& = \sum_{j=0}^{\left\lfloor \frac{i-1}{2} \right\rfloor} \binom{n-2}{i-1-2j} \binom{s+k-1-j}{n-3}  \\
B(i)^c &= \binom{n-4}{i}  \binom{n+s-i-3}{n-3}.     
\end{align*} 
By Lemma \ref{lem_binomial_identity}, we deduce that 
\begin{align*}
   \dbtilde{B}(i)^- & = A(i)^-. 
\end{align*}

Assume that $i = 2h$. Then we have
\begin{align*}
\dbtilde{B}(i)^+ &= \sum_{j=1}^{h-1} \binom{n+s-i+j-1}{n-3} \left [ \binom{n-2}{i-2j -2 } +  2 \binom{n-2}{i-1-2j}  +   \binom{n-2}{i-2j} \right ]  \\
& +  \binom{n+s-i-1}{n-3} \left ( \binom{n-1}{i-1} + \binom{n-2}{i-1} \right) \\
&+ \binom{n-2}{i} \binom{n+s-i-2}{n-3}  + \binom{n+s-h-1}{n-3}\\
&=A(i)^+ - \binom{n-2}{i} \binom{n+s-i-2}{n-4}.
\end{align*} 
Let 
\begin{align*}
    C(i)&:=p(n-1,s+2,0,i) - p(n-1,s+1,0,i) = \binom{n-3}{i} \binom{n+s-i-2}{n-4}.\\
    D(i)&:= \tilde p(n-2,s+1,0,i)=   \binom{n-4}{i}\binom{n+s-3-i}{n-4}  +  \binom{n-4}{i-1}\binom{n+s-i-2}{n-4}.
\end{align*}
To prove that $A(i) = \dbtilde{B}(i) + C(i) + D(i)$, it remains to prove that the constant terms match up. Namely,
\begin{align*}
   \binom{n-2}{i} \binom{n+s-i-1}{n-3} &= \binom{n-4}{i} \binom{n+s-i-2}{n-3} + \binom{n-4}{i-1} \binom{n+s-i-1}{n-3}  \\
   &+  \binom{n-4}{i-1}\binom{n+s-i-2}{n-3} + \binom{n-4}{i-2} \binom{n+s-i-1}{n-3}\\
   &+ \binom{n-3}{i} \binom{n+s-i-2}{n-4}\\
   & = \binom{n-3}{i} \binom{n+s-i-2}{n-3} + \binom{n-3}{i-1} \binom{n+s-i-1}{n-3}  \\
   &+ \binom{n-3}{i} \binom{n+s-i-2}{n-4}\\
   & = \binom{n-3}{i} \binom{n+s-i-1}{n-3} + \binom{n-3}{i-1} \binom{n+s-i-1}{n-3} \\
   &=\binom{n-2}{i} \binom{n+s-i-1}{n-3}.
\end{align*}
The case $i = 2h+1$ can be done in similar manner. The conclusion follows.
\end{proof}

\begin{lem}\label{lem_p_rec_2} Assume that $n \ge 4$, $s \ge 0$, and $t \ge 2$ are natural number. Then $p(n,s,t,i) - p(n,s+1,t-1,i) $ is equal to 

\begin{alignat*}{2}
  &\sum_{\ell = 0}^s \dbtilde{p} (n-2,0,\ell,i) && \text{ if } s \le t  \\
  &\sum_{\ell=0}^{t-1} \dbtilde{p} (n-2,0,\ell,i) + (s-t+1)\dbtilde{p}(n-2,0,t,i) && \\
  &+ \sum_{\ell = 1}^{s-t} (s-t+1-\ell) \left (\dbtilde{p}(n-2,\ell,t,i) -\dbtilde{p}(n-2,\ell,t-1,i) \right )&&\text{ if } s \ge t.
\end{alignat*}
\end{lem}
\begin{proof} Again, we prove the case $n$ is odd, the case $n$ is even can be done in similar manner. Let $A(i) = p(n,s,t,i) - p(n,s+1,t-1,i)$. Then we have
    \begin{align*}
        A(i)^+ &= \sum_{j=0}^{\left\lfloor \frac{i}{2} \right\rfloor} \binom{n}{i-2j} \binom{n+s-1-i+j}{n-2}\\
        A(i)^- &=\sum_{j=0}^{\left\lfloor \frac{i-1}{2} \right\rfloor} \binom{n}{i-1-2j} \binom{s+k-1-j}{n-2}\\
        A(i)^c &= -\binom{n-2}{i}\binom{n+s-i-2}{n-3}.
    \end{align*}

Let $B(s,i) = \sum_{\ell=0}^s p(n-2,0,\ell,i)$. Then we have 
\begin{align*}
    B(s,i)^+ & = \sum_{j=0}^{\left\lfloor \frac{i}{2} \right\rfloor} \binom{n-2}{i-2j}  \sum_{\ell =0}^s \binom{n + \ell -3 -i +j}{n-3} = \sum_{j=0}^{\left\lfloor \frac{i}{2} \right\rfloor} \binom{n-2}{i-2j}\binom{n + s -2 -i +j}{n-2} \\
    B(s,i)^-& = \sum_{j=0}^{\left\lfloor \frac{i-1}{2} \right\rfloor} \binom{n-2}{i-1-2j} \sum_{\ell =0}^s\binom{\ell+k-2-j}{n-3} = \sum_{j=0}^{\left\lfloor \frac{i-1}{2} \right\rfloor} \binom{n-2}{i-1-2j} \binom{s+k-1-j}{n-2}\\
    B(s,i)^c & = 0
\end{align*}
The terms $B(i)^+, B(s,i)-$ have the same form as the term $B(i)^+$ and $B(i)^-$ in Lemma \ref{lem_p_rec_1}. Hence, we deduce that 
\begin{align*}
\dbtilde{B}   (i)^+ & = \sum_{j=0}^{\left\lfloor \frac{i}{2} \right\rfloor} \binom{n}{i-2j}   \binom{n + s -1 -i +j}{n-2} - \binom{n-2}{i}\binom{n + s -2 -i}{n-3} = A(i)^+ + A(i)^c\\
    B(i)^-& = \sum_{j=0}^{\left\lfloor \frac{i-1}{2} \right\rfloor} \binom{n}{i-1-2j} \binom{s+k-1-j}{n-2} = A(i)^-.
\end{align*}
Hence, $A(i) = B(i)$ for all $i$ and $s$. Thus, it remains to consider the case $s > t$. In this case, we let $C(\ell,i) = p(n-2,\ell,t,i) - p(\ell,t-1,i)$. Then we have 
\begin{align*}
     C(\ell,i)^+&=  \sum_{j=0}^{\left\lfloor \frac{i}{2} \right\rfloor} \binom{n-2}{i-2j}   \binom{\ell + t +n  - 4 -i +j}{n-4}\\
     C(\ell,i)^-&= \sum_{j=0}^{\left\lfloor \frac{i-1}{2} \right\rfloor} \binom{n-2}{i-1-2j} \binom{\ell + t + k-3-j}{n-4}\\
     C(\ell,i)^c &=0.
\end{align*}
Thus, 
\begin{align*}
    \sum_{\ell = 1}^{s-1} C(\ell,i)^+ &= \sum_{j=0}^{\left\lfloor \frac{i}{2} \right\rfloor} \binom{n-2}{i-2j}  \left [ \binom{s +n  - 3 -i +j}{n-3} - \binom{t+n-3-i+j}{n-3} \right ] \\
    \sum_{\ell = 1}^{s-1} \ell C(\ell,i)^+ &= \sum_{j=0}^{\left\lfloor \frac{i}{2} \right\rfloor} \binom{n-2}{i-2j}  \left [ (s-t) \binom{s +n  - 3 -i +j}{n-3} - \binom{s+n-3-i+j}{n-2} \right .& \\
     &+ \left . \binom{t+n-3-i+j}{n-2} \right ].
\end{align*}
Let $D(i) = \sum_{\ell = 1}^{s-t} (s-t + 1 -\ell) C(\ell,i)$. Then we have
    
    \begin{align*}
    D(i)^+  & = \sum_{j=0}^{\left\lfloor \frac{i}{2} \right\rfloor} \binom{n-2}{i-2j}  \left [   \binom{s + n -2 -i +j}{n-2}  - (s-t) \binom{t + n - 3 -i + j}{n-3} \right .\\
    & - \left .  \binom{t + n -2 -i + j}{n-2} \right ].
\end{align*}

Let $E(i) = B(t,i) + (s-t) p(n-2,0,t,i) + D(i)$. Then we have 
\begin{align*}
    E(i)^+ &= \sum_{j=0}^{\left\lfloor \frac{i}{2} \right\rfloor} \binom{n-2}{i-2j} \binom{n + s -2 -i +j}{n-2} = B(s,i)^+\\
    E(i)^- &= B(s,i)^-.
\end{align*}
The conclusion follows.
\end{proof}
\begin{proof}[Proof of Lemma \ref{lem_main}] We prove by induction on the tuples $(n,s,t,i)$ ordered by lexicographic order with $n > t > s> i$. By Lemma \ref{lem_init_2_s_t}, Lemma \ref{lem_init_3_s_t_i}, and Lemma \ref{lem_init_n_s_0}, we have that $b(n,s,t,i) = p(n,s,t,i)$ for all boundary tuples. Thus, we may assume that $n \ge 4$ and $t \ge 1$. First, assume that $t = 0$. If $s = 0$, the conclusion is clear. Now, when $s > 0$ the conclusion follows from Lemma \ref{lem_self_rec_1}, Lemma \ref{lem_p_rec_1}, and induction. When $t > 1$, the conclusion follows from Lemma \ref{lem_self_rec_2}, Lemma \ref{lem_p_rec_2} and induction.    
\end{proof}
\begin{proof}[Proof of Theorem \ref{teo3}] By Lemma \ref{lem_main}, we have $$\beta_i(J_{n,n-2}^t) = b(n,0,t,i) = p(n,0,t,i).$$
The conclusion follows.    
\end{proof}

\begin{obs}
The Betti numbers of powers of path ideals of paths come up quite naturally if one considers the Betti numbers of general $t$-path ideals of cycles. In subsequent work, we will carry the computation for the Betti numbers of powers of path ideals of paths.
\end{obs}

\subsection*{Aknowledgments}

The second author was supported by a grant of the Ministry of Research, Innovation and Digitization, CNCS - UEFISCDI, 
project number PN-III-P1-1.1-TE-2021-1633, within PNCDI III.

\subsection*{Data availability}

Data sharing not applicable to this article as no data sets were generated or analyzed during the current study.

\subsection*{Conflict of interest}

The authors have no relevant financial or non-financial interests to disclose.

\end{document}